\documentclass{amsart}
\usepackage[utf8]{inputenc}

\usepackage{amssymb}
\usepackage{amsthm}

\usepackage[capitalise]{cleveref}		

\theoremstyle{plain}
\newtheorem{thm}{Theorem}[section]
\newtheorem{lem}[thm]{Lemma}
\newtheorem{cor}[thm]{Corollary}

\newtheorem{prop}[thm]{Proposition}

\theoremstyle{definition}
\newtheorem{dfn}[thm]{Definition}
\newtheorem{ntn}[thm]{Notation}

\newtheorem{claim}[thm]{Claim}

\newtheorem{qstn}[thm]{Question}

\newcommand{\converge}{\!\!\downarrow}
\newcommand{\cat}{\widehat{\phantom{\alpha}}}
\newcommand{\uh}{\!\!\upharpoonright\!}

\newcommand\seq[1]{\langle #1 \rangle}

\renewcommand\*[1]{\mathbf{#1}}
\renewcommand\=[1]{\text{#1}}

\begin{document}

\title{Limit Complexities, Minimal Descriptions, and $n$-Randomness}

\author{Rodney Downey}

\address{School of Mathematics and Statistics\\
Victoria University \\
PO Box 600, Wellington, 6140\\
New Zealand}

\author{Lu Liu}

\address{School of Mathematics and Statistics,  HNP-LAMA\\
Central South University\\
Changsha, Hunan Province,\\
China. 410083
}

\author{Keng Meng Ng}

\address{Mathematics Department\\
Nanyang University of Technology\\
Singapore}

\author{Daniel TuretskY}

\address{School of Mathematics and Statistics \\
Victoria University \\
PO Box 600, Wellington, 6140\\
New Zealand}

\thanks{All supported by the Marsden Fund of New Zealand, with Lu on a 
Postdoctoral Fellowship. We thank Denis Hirschfeldt, Joeseph Miller and Andr{\'e} Nies for helpful discussions.}
\maketitle

\begin{abstract}
Let $K$ denote prefix-free Kolmogorov Complexity, and $K^A$ denote it relative to an oracle $A$. We show that for any $n$, $K^{\emptyset^{(n)}}$ is definable 
purely in terms of the unrelativized notion $K$. It was already 
known that 2-randomness is definable in terms of $K$ (and plain complexity $C$)
as those reals which infinitely often have maximal complexity. We can use our characterization
to show that $n$-randomness is definable purely in terms of $K$.
To do this we extend a certain ``limsup'' formula from the literature, and 
apply Symmetry of Information. This extension entails a novel use of 
semilow sets, and a more precise analysis of the complexity of 
$\Delta_2^0$ sets of mimimal descriptions.
\end{abstract}

\section{Introduction}
\subsection{$n$-randomness}
The concern of this paper is 
Kolmogorov complexity, where $C$ denotes plain complexity and 
$K$ prefix-free complexity. 
A fundamental theorem in the theory of algorithmic randomness 
is Schnorr's Theorem (see Chaitin \cite{58}) that a real $X$ 
is Martin-L\"of random iff $K(X\upharpoonright n)\ge^+ n$ for all
$n$.
We know that this result relativises, and hence 
a real $X$ is $A$-random iff
$K^A(X\upharpoonright n)\ge^+ n$ for all
$n$.
In particular, if $A=\emptyset^{(n)}$, then 
this says that $X$ is $n+1$-random iff 
$K^{\emptyset^{(n)}}(X\upharpoonright n)\ge^+ n$, for all $n$.

The reason we need to use prefix-free complexity in the definition 
of $k$-randomness is that, as proven by Martin-L\"of, there is 
no real $X$ such that $C(X\upharpoonright n)\ge ^+ n$ for all $n$,
due to complexity oscillations (Martin-L\"of \cite{ML66,ML71})
and the failure of $C$ to capture the intentional meaning of least descriptions. (See, e.g.\ Downey and Hirschfeldt \cite{DH},
Chapter 6, or Nies \cite{Ni09}.)
One of the striking results in this theory is 
there are reals where  $C(X\upharpoonright n)\ge ^+ n$  for infinitely many
$n$, and these coincide with one of the randomness classes
(and indeed this also holds for an analogous 
fact about strong $K$-randomness\footnote{A string $\sigma$ is 
called \emph{strongly} $K$-random if it achieves maximal
$K$-complexity, namely $K(\sigma)=^+ |\sigma|+K(|\sigma|).$}), as we see below.

\begin{thm}
\label{the1} \ 
\begin{itemize}
\item ~{\em [Miller \cite{Mill}, Nies, Stephan, and Terwijn \cite{NST}]}
$X$ is 2-random iff
\[
C(X\uh n)\ge ^+ n
\]
for infinitely many $n$.
\item ~{\em [Miller \cite{Mill2}]}
$X$ is 2-random iff 
\[
K(X\uh n)\ge ^+ n +K(n)
\]
for infinitely many $n$.
\end{itemize}
\end{thm}
Thus, there is a definition of being $2$-random involving 
the basic definition of $K$ or $C$ which does not involve any 
relativization. 

One goal in this paper is to give a definition of $n$-randomness for all 
$n\in \omega$ only involving $K$, by giving a definition
of $K^{(n)}$ purely in terms of $K$. 
Our starting point is the following attractive result.

\begin{thm}[Bienvenu, Muchnik, Shen and Vereshchagin~\cite{Bei}]\label{thm:motivation}
\[
K^{\emptyset'}(\sigma)=^+\limsup_n K(\sigma \mid n).
\]
\end{thm}
The same result holds for $C$ in place of $K$.

Here we remind the reader that 
$K(\sigma\mid n)$ is the prefix-free Kolmogorov complexity
of $\sigma$ when $\overline{n}$, a self-delimited version of $n$,
is provided as an oracle.
The reader might think that 
this result provides a definition  of 
$K^{\emptyset'}$ in terms of $K$, but 
$K(\sigma\mid n)$ is not an unrelativized notion. 
Indeed, in \S \ref{misc}, we will see that 
finite strings can have very strong compression power.
In \S \ref{misc}, we will also give a partial analysis 
as to precisely for which $n$, the limsup is achieved.

Nevertheless, our plan is to leverage 
$K(\sigma\mid n)$, and we will do this using 
Symmetry of Information (Levin and G\'acs \cite{Ga}, Chaitin \cite{58})
which  says that 
\[
K(\sigma,n)=^+ K(n)+K(\sigma \mid n^*).
\]
Here the reader should recall that 
$K(\nu,\rho)$ is the complexity of the 
pair $\langle \nu,\rho\rangle$, and
that for any string $\tau$, $\tau^*$ is the first to occur of length $K(\tau)$
with $U(\tau^*)=\tau$. (In the case of $C$
we will write $\tau^*_C$.)

While $\tau^*$ is a particular minimal code for $\tau$ (the first to appear), it may not be the only code for $\tau$ of length $K(\tau)$.  We will also be interested in all minimal codes, and so we adopt the following notation.

\begin{ntn}
For a universal machine $U$ (prefix-free or otherwise), let $N_U = \{ n^* :  n \in \omega\}$, where $n^*$ is defined based on $U$.  Let $M_U = \{ \rho : U(\rho)\converge \wedge |\rho| = K(U(\rho))\}$.

When $U$ is clear from context, we will omit the subscript.
\end{ntn}
Observe that $N_U \subseteq M_U$.

In \S \ref{main}
we will prove the following.

\begin{thm}\label{thm:main_result}
For any universal prefix-free machine $U$, 
\[
\limsup_n K(\sigma \mid n^*) = \limsup_{\tau \in M_U} K(\sigma\mid \tau) = K^{\emptyset'}(\sigma).
\]
\end{thm}

Notice that by rearranging Symmetry of Information, we obtain $K(\sigma\mid n^*) =^+ K(\sigma, n) - K(n)$.  Hence $K^{\emptyset'}(\sigma) =^+ \limsup_n [K(\sigma, n) - K(n)]$, giving a definition of $K^{\emptyset'}$ purely in terms of $K$ without relativization.

By relativizing \Cref{thm:motivation}, we obtain 
\[
K^{A'}(\sigma) =^+ \limsup_n K^A(\sigma \mid n)\footnote{Indeed 
$=^+\limsup_n K(\sigma \mid A\uh n)$.},
\]
for all $A$.  By appropriately iterating and relativizing \Cref{thm:main_result}, we obtain a definition of $K^{\emptyset^{(n)}}$ for all $n \in \omega$, e.g.,
\begin{align*}
K^{\emptyset^{(2)}}(\sigma) &=^+ \limsup_n [K^{\emptyset'}(\sigma, n) - K^{\emptyset'}(n)]\\
&=^+ \limsup_n \left(\limsup_m [K(\sigma, n, m) - K(m)] - \limsup_m [K(n,m) - K(m)]\right).
\end{align*}
From this follows a definition of $n$-randomness purely in terms of unrelativized $K$.

\subsection{The complexity of the sets $M_U$ and $N_U$}
The question arises how should we prove 
the $\limsup_n K(\sigma \mid n^*)$ theorem? The answer really comes from understanding the 
behaviour of the set of minimal descriptions.  
One hint came from unpublished work of Hirschfeldt:

\begin{thm}[Hirschfeldt, unpubl.]
\label{hir}
$C^{\emptyset'}(\sigma)=^+ \limsup_n C(\sigma \mid n^*_C)$.
\end{thm}

We will prove this result in \S \ref{denis}. The method is 
to construct an infinite low subset of $N_V$, for $V$ the machine generating $C$, and use some relativization tricks. 

We had hoped to use this method for $K$, but unfortunately we 
were able to prove a result saying that this is impossible.

\begin{thm}
\label{complete}
Let $U$ be a universal prefix-free machine, and let $S$ be an infinite $\Delta_2^0$ subset of $M_U$. Then $\emptyset'\le_T S$.
\end{thm}

Note that $M_U$ and $N_U$ are of degree $\*{0}'$ e.g.\ ~\cite{DH}.
For example, $N_U\ge_{wtt} \Omega$ by the Coding Theorem (\cite{DH}, \S 3.9), and 
$\Omega$ is wtt-complete.
   In view of Theorem \ref{complete},
it would seem reasonable to suggest that they
are introreducible, which in this context would mean that every infinite subset
computes $\emptyset'$.
However, Joeseph Miller proved that this is not the case. We include this
also in \S \ref{colour}.

We remark that the proof can also be adapted to show that 
an infinite $\Delta^0_2$ hitting set for a Solovay function\footnote{That is, 
a computable $F$
such that $F(\sigma)\ge^+ K(\sigma)$ for all $\sigma$,
and $F(\sigma)=^+K(\sigma)$ for infinitely many $\sigma$.}
must also be Turing complete.
We prove these results, which are on independent interest,
in \S \ref{colour}.

In the end we found a way around these problems using an 
idea from studies in the automorphism group of the lattice of 
computably enumerable sets (Soare \cite{Soa}.)
We recall that a set $S$ is called 
\emph{semi-low} iff 
\[
\{e\mid S\cap W_e \ne \emptyset\}\le_T \emptyset',
\]
i.e., a pointwise version of being low.
This notion was introduced as a method towards characterizing when 
the lattice of supersets of a c.e.\ set was isomorphic to 
the lattice of all c.e.\ sets.
In \S \ref{main}, we will show that, although there cannot be an infinite low 
subset of $M_U$ or $N_U$, there \emph{can} be a semi-low one.
In \S \ref{semi} we will prove that this is enough for our 
main result.
There have been other uses of semi-lowness outside of the lattice of c.e.\ sets,
such as Downey and Melnikov \cite{DM} in the study of abelian groups,
but these seem sporadic at best.

\subsection{The machine existence theorem}

As we will be using it several times, we recall the machine existence theorem and fix our notation surrounding it.

\begin{dfn}
For a set $A \subseteq 2^{<\omega} \times \omega$, $\={wt}(A) = \sum_{(\sigma, s) \in A} 2^{-s}$.
\end{dfn}

\begin{thm}
[KC Theorem or The Machine Existence Theorem~see Downey and Hirschfeldt \cite{DH}, \S 3.6]
If $A \subseteq 2^{<\omega}\times \omega$ is c.e.\ and has $\={wt}(A) \le 1$, then there is a prefix-free machine $V$ such that for every $(\sigma, s) \in A$, there is a $\rho$ with $|\rho| = s$ and $V(\rho) = \sigma$.  Further, an index for $V$ can be effectively obtained from a c.e.\ index for $A$.
\end{thm}

\begin{cor}\label{cor:finite_request_sets}
If $A \subseteq 2^{<\omega}\times \omega$ is c.e.\ with $\={wt}(A) < \infty$, then for all $(\sigma, s) \in A$, $K(\sigma) \le^+ s$.
\end{cor}
Sets $A$ of this form are sometimes called {\em request sets}.

The following is more of a proof technique, but we will state it as a corollary.
\begin{cor}\label{cor:we_control_part_of_U}
As part of a uniform construction, we may effectively obtain an $\epsilon > 0$ such that if we enumerate $A \subseteq 2^{<\omega} \times \omega$ with $\={wt}(A) \le \epsilon$, then $K(\sigma) \le s$ for every $(\sigma,s) \in A$ (observe the lack of additive constant).
\end{cor}

\begin{proof}
 Fix $U$ the universal prefix-free machine used to define $K$. We will generate an auxiliary c.e.\ set $B$ with $\={wt}(B) \le 1$.  By the Recursion Theorem, we know a c.e.\ index for the set $B$ we will enumerate.  By the Machine Existence Theorem, this effectively gives us an index for a prefix-free machine $V$. From this we effectively obtain a string $\rho$ with $U(\rho\tau) = V(\tau)$ for all $\tau$, and so $K(\sigma) \le t + |\rho|$ for all $(\sigma, t) \in B$.
 
 Set $\epsilon = 2^{-|\rho|}$, and define $B$ by enumerating $(\sigma, s-|\rho|)$ whenever $A$ enumerates $(\sigma, s)$, provided this enumeration does not put $\={wt}(B)$ over 1.  If $\={wt}(A) \le \epsilon$, then $(\sigma, s-|\rho|)$ is enumerated into $B$ for every pair $(\sigma, s) \in A$, and so $K(\sigma) \le (s-|\rho|) + |\rho| = s$, as desired.
\end{proof}

\section{Hirschfeldt's Theorem}
\label{denis}

We prove \cref{hir}.
We will prove that 
\[
C^{\emptyset'}(\sigma)=\limsup_n C(\sigma \mid n^*_C).
\]
Consider the $\Pi_1^0$ class
of sequences 
\[
P=\{(m_0, m_1, \dots) \mid \forall n\, [2^n \le m < 2^{n+1} \wedge C(m_n) \ge n]\}.
\]
A simple counting argument shows that there is an appropriate $m_n$ for every $n$, and so $P$ is nonempty. Since there are only $2^n$ options for $m_n$, $P \subseteq 2^\omega$ under an appropriate effective identification.  So there is a low infinite path $L = (m_0, m_1, \dots)$ by the Low Basis Theorem.

Recall that for $2^n \le m < 2^{n+1}$, $C(m) \le^+ n$, so fix the least $d$ such that $\exists^\infty n\, C(m_n) = n+d$, and fix an $N$ such that $C(m_n) \ge n+d$ for all $n \ge N$.  Then
\[
X = \{ (m_n)^*_C : n \ge N \wedge C(m_n) = n+d\}
\]
is $L$-c.e.\ and an infinite subset of $N_V$, where $V$ is the universal machine defining $C$.  Fix $(\rho_i)_{i \in \omega}$ an $L$-computable enumeration of $X$.

Then by relativizing \Cref{thm:motivation},
we have
\[
C^{\emptyset'}(\sigma) =^+ C^{L'}(\sigma) =^+ \limsup_i C^L (\sigma \mid i).
\]
Note that we can $L$-effectively pass between $i$ and $\rho_i$, so $C^L(\sigma\mid i) =^+ C^L(\sigma\mid\rho_i)$, giving
\[
C^{\emptyset'}(\sigma) =^+ \limsup_{\rho \in X} C^L(\sigma\mid \rho).
\]
We also have
\[
\limsup_{\rho \in X} C^L(\sigma\mid \rho) \le^+ \limsup_{\rho \in X}C(\sigma\mid\rho),
\]
since oracles can only help;
\[
\limsup_{\rho \in X}C(\sigma\mid\rho) \le \limsup_{n^*_C} C(\sigma\mid n^*_C) \le \limsup_{\rho \in M_V} C(\sigma\mid \rho) \le \limsup_{\rho \in 2^{<\omega}}C(\sigma\mid \rho),
\]
since $X \subseteq N_V \subseteq M_V \subseteq 2^{<\omega}$, and limit supremums over larger sets are larger; and finally
\[
\limsup_{\rho \in 2^{<\omega}}C(\sigma\mid \rho) =^+ C^{\emptyset'}(\sigma),
\]
by the unrelativized version of \Cref{thm:motivation}, after an effective identification of $2^{<\omega}$ with $\omega$.  Picking out the relevant bits, we see that 
\[
C^{\emptyset'}(\sigma)=^+ \limsup_n C(\sigma\mid n^*_C) =^+ \limsup_{\rho \in M_V} C(\sigma\mid \rho).
\]

\section{No low hitting sets}
\label{colour}

We prove Theorem \ref{complete}.  Fix a universal prefix-free machine $U$, and suppose that $X=\lim_s X_s$ is an infinite $\Delta_2^0$ subset of $M_U$.  Fix $\epsilon$ as in \cref{cor:we_control_part_of_U}; we will enumerate an appropriate set $A \subseteq 2^{<\omega} \times \omega$.  

We describe how we code whether $n\in \emptyset'$.  Fix $k \in \omega$ with 
\[
\frac1k < 2^{-(n+2)}\epsilon.
\]
To do this coding, we will define a $k$-colouring $\chi$ on $\={dom}(U)$.  This colouring will be unique to $n$; the colourings for other values of $n$ will have no interaction.

We declare that colour $i$ is \emph{small} if
\[
\sum_{\chi(\sigma)=i}2^{-|\sigma|}\le \frac{1}{k}.
\]
This has natural approximations: at a stage $s$, based on the finitely many strings we have so far coloured, a colour may still be small or may have already proven itself to be large.  Note that since colours are disjoint, and we have $k$ colours, there is always at least one small colour.

Suppose that $\sigma$ enters $\mbox{dom}(U)$ at stage $s$.
Let
\[
r_i=\min\{|\tau| \mid 
\chi(\tau)=i \land \tau\in X_s\}.
\]
We regard this as infinite if there is no such $\tau$.
Then amongst the $i$ which are small at stage $s$, we fix a $j$ maximizing $r_j$ and colour $\chi(\sigma) = j$.

Suppose that $n$ enters $\emptyset'$ at some stage $s+1$.  Fix a single colour $j$ which was small at stage $s$; we invalidate all the strings which had colour $j$ at stage $s$.  That is, for every $\sigma \in \={dom}(U)[s]$ with $\chi(\sigma) = j$, we enumerate $(U(\sigma), |\sigma|-1)$ into $A$.  Provided $\={wt}(A) \le \epsilon$, this will ensure that $\sigma \not \in M_U$ for each such $\sigma$.

As $j$ was small at stage $s$, the weight of these pairs is at most $\frac2k < 2^{-(n+1)}\epsilon$.  Thus, summing over the strategies for every $n$, $\={wt}(A) \le \sum_n 2^{-(n+1)}\epsilon = \epsilon$, as required.

\begin{claim} $X$ has members of every small colour.
\end{claim}

\begin{proof}
 Fix $A$ the set of small colours which occur in $X$, and suppose this is not all the small colours.
 
 Fix a length $n$ such that every colour in $A$ occurs on a string $\tau \in X$ with $|\tau| \le n$, and fix $t$ sufficiently large such that $X$ has converged on strings of length at most $n$ by stage $t$, i.e., if $|\tau| \le n$, then for all $s \ge t$, $X_s(\tau) = X(\tau)$.  Suppose also that $t$ is large enough such that every large colour has proven itself large by stage $t$.
 
 Since $X$ is infinite, it contains some $\tau$ which enters $\={dom}(U)$ at some stage $s > t$.  At stage $s$, $r_i \le n$ for every $i \in A$, whereas $r_i > n$ for every small colour $i \not \in A$.  So $\chi(\tau)$ will be a small colour not in $A$, contrary to choice of $A$.
\end{proof}
We can now state our procedure for computing $\emptyset'(n)$ from $X$: in the colouring for $n$, search for a stage $t$ such that for every colour $i$ which is still small at stage $t$, some element of $X$ has been given colour $i$ by stage $t$; then output $\emptyset'_t(n)$.

As just argued, there is eventually some stage at which $X$ intersects every small colour, so this algorithm is total.  Suppose first that $n \not \in \emptyset'$.  Then certainly $\emptyset'_t(n) = 0$, as desired.

Suppose instead that $n \in \emptyset'$, and fix the stage $s+1$ at which it enters.  Fix the chosen colour $j$.  Then no $\sigma$ which received colour $j$ at or before stage $s$ belongs to $M_U$, and so cannot belong to $X$.  Since colour $j$ is still small by stage $s$, $j$ witnesses $t \not \le s$, giving $\emptyset'_t(n) = 1$, as desired.

This concludes the proof of \cref{complete}.

\smallskip

The same method can be used to prove:

\begin{cor} \label{corc}
Suppose that $X$ is an infinite $\Delta_2^0$ set of 
hitting points for a Solovay function $F$. That is, a set $S$ of 
points $n$ where $F(n)=^+ K(n)$.
Then $\emptyset'\le_T S$.
\end{cor}

We remark that Corollary \ref{corc} improves a result of 
Bienvenu, Downey, Merkle and Nies
\cite{BDMN}
who showed that the collection of \emph{all} hitting points is 
Turing complete.

The reader should note that if $\sigma=m^*$, then $\sigma$ must 
be weakly $K$-random in that $K(\sigma)\ge^+ |\sigma|$.
The reason is that if $K(\sigma)<< |\sigma|$ then using the 
KC Theorem, we can use $\sigma^*$ to describe $m$, in a machine $M$ 
we build. This would show that $K(m)<<|\sigma|=|m^*|$, a contradiction.
This brings in to focus the question of precisely which 
weakly $K$-random strings are minimal descriptions.
By the Low Basis Theorem, there are infinite low collections 
of weakly $K$-random strings.
At most finitely many can be minimal descriptions.
Another consequence of Theorem \ref{complete} is the following\footnote{More or less the same proof will also give this for intersections of 
$\Delta_2^0$ sets of hitting points for Solovay functions, and 
$\Delta_2^0$ subsets of $M^*,$ this last one by the Coding Theorem,
there are at most $O(1)$ many elements of $M^*$ of length $n^*$ for a fixed $n$.}.

\begin{cor}
If $X$ is a $\Delta^0_2$ collection of weakly $K$-random strings (that is, 
$K(\sigma)\ge^+ |\sigma|$), and $|X \cap N_U| = \infty$, then $X$ computes $\emptyset'$.
\end{cor}

\begin{proof}
Fix $d$ such that for every $\sigma \in X$, $K(\sigma) > |\sigma|-d$.  For every $n$, let $n^*_s$ be the natural stage $s$ approximation to $n^*$.  This may be undefined for small $s$, but it will eventually converge to the true $n^*$.  Further, if $n^*_s\converge$ and $n^*_{s+1} \neq n^*_s$, then $|n^*_{s+1}| < |n^*_s|$.

Again fix $\epsilon$ as in \cref{cor:we_control_part_of_U}.  Fix $k$ with $2^{-k} < \epsilon$.  For every $n$ and $s$ with $|n_s^*| > |n^*| + k+d$ (a c.e.\ event), we enumerate $(n, |n_s^*| - d)$ into $A$.  Since for a fixed $n$ there is at most one $n^*_s$ of any given length, the weight of our requests is bounded by
\[
\sum_{n} \sum_{i > |n^*|+k+d} 2^{-(i-d)} = 2^{-k} \sum_n 2^{-|n^*|} < 2^{-k}.
\]
Thus $\={wt}(A) \le \epsilon$, and so $K(n^*_s) \le |n^*_s| - d$ for every such $n^*_s$.

It follows that if $n^*_s \in X$, then since $K(n^*_s) > |n^*_s| - d$ by assumption, $|n^*_s| \le |n^*| + k+d$, or $|n^*| \ge |n^*_s| - k - d$.  As in the proof of \cref{hir}, this allows $X$ to enumerate an infinite sequence from $N_U$.  Since every infinite c.e.\ set contains an infinite computable set, and this relativizes, we get that $X$ computes an infinite $Y \subseteq N_U$.  As $X$ is $\Delta^0_2$, $Y$ is as well, so $Y \ge_T \emptyset'$, and thus $X \ge_T \emptyset'$.
\end{proof}

As we mentioned in the introduction, Thorem \ref{complete}
cannot be improved to show that all infinite subsets of 
$N_U$ compute $\emptyset'$.

\begin{thm}[Joeseph Miller, unpublished]
There is an infinite $X \subset N_U$ which does not compute 
$\emptyset'$.
\end{thm}

\begin{proof}
Let $P$ be a
bounded $\Pi^0_1$ class of $K$-compression functions\footnote{A $K$-compression
function is an injective function $G: \omega \to \omega$ such that for all $n$,
$G(n)\le K(n)$. These were introduced by Nies, Stephan and Terwijn
\cite{NST} in their proof that 2-randomness is the same as 
infinitely often $C$-random.}.  Since we have an a priori upper bound of $K(n) \le^+ 2\log(n)$, we may take $P \subseteq 2^\omega$.
 Let $F$ be a
``weakly-low for $K$'' path. 
That is, there are infinitely many $n$ with $F(n)=^+K(n)$. This 
can be shown to exist using
the  ``low for $\Omega$''-Basis Theorem\footnote{Every $\Pi_1^0$ 
class on $2^\omega$ contains a $\emptyset'$-left c.e.\ real $A$
relative to which $\Omega^A=\Omega$.}  
(Downey, Hirschfeldt, Miller, Nies \cite{DHMN}, Reimann and Slaman \cite{RS})
and the fact that low for $\Omega$  is equivalent to ``weakly
low for $K$'', see Downey and Hirschfeldt \cite{DH}.

Now fix the least $c$ with $K(n) = F(n)+c$ for infinitely many $n$, and fix an $m$ with $K(n) \ge F(n) + c$ for all $n \ge m$. $F$ can enumerate an infinite subset of $N_U$: $\{ n_s^* : n \ge m \wedge |n_s^*| = F(n)+c\}$.  Thus $F$ computes an infinite $X \subseteq N_U$ (again relativizing the fact that every infinite c.e.\ set has an infinite computable subset), Since $F$ does not compute $\emptyset'$ (since it is weakly-low for $K$), $X$ also does
not compute $\emptyset'$.\end{proof}

\section{Conditional complexity along semi-low sets}
\label{semi}

Semi-lowness has previously been studied for co-c.e.\ sets.  We are interested in it for $\Delta^0_2$ sets, in which case it is not entirely clear that the following is the correct definition\footnote{An alternative definition would additionally require that $\{e : W_e \subseteq X\}$ is $\Delta^0_2$; note that when $X$ is co-c.e., this set is $\Pi^0_1$.}, but it is the definition relevant to our current interest.

\begin{dfn}
Let $(W_e)_{e \in \omega}$ be a standard listing of c.e.\ sets.  A set $X$ is {\em semi-low} if the set $\{ e : X \cap W_e \neq  \emptyset\}$ is $\Delta^0_2$.
\end{dfn}

Recall \Cref{thm:motivation}:
\[
K^{\emptyset'}(\sigma) =^+ \limsup_{n \in \omega} K(\sigma\mid n).
\]

As we have seen, it can be helpful to consider $\limsup_{n \in X} K(\sigma\mid n)$ for $X$ infinite.  It is immediate that this is $\le^+ K^{\emptyset'}(\sigma)$, as we are taking a limit supremum over a smaller set.  It turns out that for semi-low sets, we have equality.

\begin{prop}\label{prop:semi-lowness_suffices}
If $X$ is semi-low and infinite, then $K^{\emptyset'}(\sigma) =^+ \limsup_{n \in X} K(\sigma\mid n)$.
\end{prop}

\begin{proof}
As one direction is immediate, it remains to show that
\[
K^{\emptyset'}(\sigma) \le^+ \limsup_{n \in X} K(\sigma\mid n).
\]
We will work with request sets.

For each $n \in \omega$, define $A_n = \{ (\sigma, s) : s \ge K(\sigma\mid n)\}$.  We may think of $A_n$ as the request set generating $K(\cdot |n)$.  Observe that $\={wt}(A_n) < 2$ for all $n$.

Note that for any finite set $D \subset 2^{<\omega} \times \omega$ and any $m \in \omega$, the set $\{ n\ge m : \={wt}(D \cup A_n) > 2\}$ is c.e.\ (it is even primitive recursive with appropriate assumptions on $K(\cdot|n)$, but this is not necessary).  Indeed this is uniform, so we may fix a total computable function $e$ such that $W_{e(D, m)} = \{n \ge m : \={wt}(D\cup A_n) > 2\}$, where $D$ is given by a canonical index.

We will build a $\emptyset'$-enumerable request set $B$ with $\={wt}(B) \le 2$ and such that for all $\sigma$, if $s = \limsup_{n \in X} K(\sigma\mid n)$, then $(\sigma, s) \in B$.  By \cref{cor:finite_request_sets} relativized to $\emptyset'$, this will suffice to prove the result.

Fix an effective listing $(\sigma_m, s_m)_{m  \in \omega}$ of $2^{<\omega}\times \omega$ such that every pair is repeated infinitely many times on the list.  We define $B$ as follows:
\begin{itemize}
    \item $B_0 = \emptyset$;
    \item Given $B_m$, fix $D = B_m \cup \{(\sigma_m, s_m)\}$.  If $X \cap W_{e(D,m)} = \emptyset$, we let $B_{m+1} = D$; otherwise, we let $B_{m+1} = B_m$.
\end{itemize}
As $X$ is semi-low, $\emptyset'$ can run this construction, and so $B$ is $\emptyset'$-enumerable.

\begin{claim}
For all $n \ge m$ with $n \in X$, $\={wt}(B_m \cup A_n) \le 2$, and thus $\={wt}(B) \le 2$.
\end{claim}

\begin{proof}
Suppose not.  Then as this clearly holds for $B_0$, we may fix $m+1$ the least value where the claim is violated.  So there is some $n \ge m+1$ with $n \in X$, $\={wt}(B_m \cup A_n) \le 2$ and $\={wt}(B_{m+1} \cup A_n) > 2$.  As $B_{m+1} \neq B_m$, we must be in the case $X \cap W_{e(D, m)} = \emptyset$, with $B_{m+1} = B_m \cup \{(\sigma_m, s_m)\} = D$.  But $n \in X \cap W_{e(D, m)}$, a contradiction.

That $\={wt}(B) \le 2$ then follows from $X$ being infinite.
\end{proof}

\begin{claim}
For any $\sigma$, if $s = \limsup_{n \in X} K(\sigma\mid n)$, then $(\sigma, s) \in B$.
\end{claim}

\begin{proof}
Fix an $n_0$ such that for all $n \ge n_0$ with $n \in X$, $K(\sigma\mid n) \leq s$.  Then for all $n \ge n_0$ with $n \in X$, $(\sigma, s) \in A_n$.  Fix an $m \ge n_0$ such that $(\sigma, s) = (\sigma_m, s_m)$.  Let $D = B_m \cup \{\sigma_m, s_m\}$.  Then for all $n \ge m$ with $n \in X$, $D \cup A_n = B_m \cup A_n$, and $\={wt}(B_m \cup A_n) \le 2$.  So $X \cap W_{e(D,m)} = \emptyset$, and $(\sigma, s) \in B_{m+1}$ by construction.
\end{proof}

This completes the proof.
\end{proof}

\section{Conditional complexity along minimal codes}
\label{main}

Fix a universal prefix-free machine $U$.  We are interested in $\limsup_{\tau \in M_U} K(\sigma\mid \tau)$ and $\lim_{\tau \in N_U} K(\sigma\mid \tau)$.  First we verify that these values are machine independent.

\begin{lem}
If $U$ and $V$ are universal prefix-free machines, and $K(\cdot | \cdot)$ is defined from a third (unnamed) universal prefix-free machine, then
\[
\limsup_{\tau \in M_U} K(\sigma\mid \tau) =^+ \limsup_{\tau \in M_V} K(\sigma\mid \tau) = \limsup_{\tau \in N_V} K(\sigma\mid \tau).
\]
\end{lem}

\begin{proof}
By symmetry, and the fact that $N_V \subseteq M_V$, it suffices to show
\[
\limsup_{\tau \in M_U} K(\sigma\mid \tau) \le^+ \limsup_{\tau \in N_V} K(\sigma\mid \tau).
\]
By standard arguments, there is a constant $c$ such that if $\tau \in M_U$, $\rho \in M_V$, and $U(\tau) = V(\rho)$, then $| |\tau| - |\rho| | \le c$.

For each $\tau \in 2^{<\omega}$ and $i \in \mathbb{Z}$ with $|i| \le c$, let $ \rho(\tau,i)$ be the first $\rho$ located with $|\rho| = |\tau|+i$ and $U(\tau)\converge = V(\rho)\converge$, if such $\rho$ exists.  We define
\[
B_\tau = \{ (\sigma, s) : \exists i\, \rho(\tau,i)\converge \wedge s \ge K(\sigma\mid \rho(\tau,i))\}.
\]
Then $\={wt}(B_\tau) \le \sum_{|i| \le c} \sum_{\sigma} 2\cdot 2^{-K(\sigma\mid \rho(\tau,i))} \le (2c+1)\cdot 2$, and thus these are uniformly given request sets.

It follows that $K(\sigma\mid \tau) \le^+ K(\sigma\mid \rho(\tau,i))$ for all $i$ with $\rho(\tau,i)\converge$.  Note that if $\tau \in M_U$, then there is an $i$ with $\rho(\tau,i)\converge \in N_V$.  The claim follows.
\end{proof}

\begin{prop}\label{prop:semi_low_existence}
There is a universal prefix-free machine $U$ and an infinite, semi-low set $X \subseteq N_U$.
\end{prop}

\begin{proof}
Fix some universal prefix-free machine $V$.  
We define $U(0^{3}\cat\tau) = V(\tau)$ for all $\tau$, which makes $U$ universal while giving us the freedom to do as we like on other neighborhoods.

Let $(N_s)_{s \in \omega}$ be the natural approximation to $N_U$.  We will have semi-lowness requirements $R_e$ and infiniteness requirements $P_n$.  The strategy for each requirement will {\em claim} various strings in $N_s$, and each strategy will have a directive at every stage: meet or avoid.  A string may only be claimed by a single strategy at a time, and a strategy will retain its claim on a string until either the string leaves $N_s$, or a higher priority strategy claims the string.  In either case, the strategy will immediately relinquish its claim.

We will build a c.e.\ set $A \subseteq 2^{<\omega} \times\omega$.  As we will argue, the sum $\sum 2^{-|\sigma|}$ over all strings $\sigma$ which are ever claimed in the construction will be bounded by 1/2.  The first time a string $\sigma$ is claimed by a strategy (i.e.\ it was unclaimed at all previous stages), we will immediately enumerate $(k, |\sigma|-1)$ into $A$, where $k$ is larger than any value yet seen in the construction.  As the previous sum is bounded by 1/2, $\={wt}(A) \le 1$.  By the Machine Existence Theorem, we uniformly obtain the index of a corresponding prefix-free machine $Q$ such that for every such pair $(k, |\sigma|-1) \in A$, there is a $\tau$ with $|\tau| = |\sigma|-1$ and $Q(\tau) = k$.

We define $U(1\cat\tau) = Q(\tau)$ for all $\tau$.  Suppose $\sigma$ is first claimed at stage $s$, and so we enumerate $(k, |\sigma|-1)$ into $A$ for some large $k$.  Then for the appropriate $\tau$, $|1\cat\tau| = |\sigma|$ and $U(1\cat\tau) = k$.  By the largeness of $k$, $N_s$ contains no codes for $k$, so $\tau$ will belong to $N_U$ unless $V$ enumerates a sufficiently shorter code at some subsequent stage.  The idea is that whenever a potential element of $N_U$ is claimed, we ensure it is replaced with a new element of the same length or shorter.

This completes the description of $U$, apart from describing how strategies claim strings.  We order our requirements $R_0, P_0, R_1, P_1, \dots$.  At stage $s$, we consider the first $s$ requirements in order, implementing the following strategies.

\smallskip

{\em Strategy for $P_n$:}

$P_n$ will always have the {\em meet} directive, and will claim at most one string at a time.  At stage $s$, if it retains a claimed string from the previous stage (i.e.\ $s>0$, $P_n$ had a claimed string at stage $s-1$, that string remains in $N_s$, and that string has not been claimed by a higher priority strategy earlier in stage $s$), then we take no further action.  Otherwise, if there is a string in $N_s$ of length at least $2n+4$ and unclaimed by any higher priority strategy, $P_n$ claims the least one (in some effective ordering).  If $P_n$ did not retain a claimed string, and there is no appropriate string to claim, we simply do nothing.

\smallskip

{\em Strategy for $R_e$:}

At stage $s$, let $C$ be the set of strings which $R_e$ retains the claim on from the previous stage: those strings which it claimed at stage $s-1$, which remain in $N_s$, and which were not claimed by a higher priority strategy earlier in stage $s$.  If $R_e$ was not considered at stage $s-1$ (possibly because $s = 0$), then $C=\emptyset$.

Let $w = \sum_{\sigma \in C} 2^{-|\sigma|}$.  Our directive for $R_e$ at stage $s$ will depend on $w$ and on $R_e$'s directive at stage $s-1$:
\begin{itemize}
    \item If $w = 0$ (i.e.\ $C = \emptyset$), then $R_e$ has the directive {\em avoid} at stage $s$.
    \item If $0 < w \le 2^{-(2e+4)}$, then $R_e$ retains the same directive as it had at the previous stage ($w > 0$ entails that $R_e$ was considered at the previous stage).
    \item If $w > 2^{-(2e+4)}$, then $R_e$ has the directive {\em meet} at stage $s$.
\end{itemize}

If our directive for $R_e$ at stage $s$ is {\em meet}, we take no further action at this stage.

If our directive for $R_e$ at stage $s$ is {\em avoid}, and there is a string in $N_s \cap W_{e,s}$ of length at least $2e+4$ and unclaimed by any higher priority strategy, then $R_e$ claims the least such string (in some effective ordering).  Our action for $R_e$ at this stage is then complete.

Thus the behaviour of $R_e$ is a yo-yo: it continues to claim strings until the claimed strings surpass its threshold, at which point it stops and lets those strings bleed away.  Once it has lost all of its claims, the strategy begins claiming new strings again.  Note that we claim no more than one string for $R_e$ at each stage.

This completes the construction.

\smallskip

{\em Definition of $X$:}

Observe that if $\sigma \in N_U$ is claimed by a strategy, then there are only two possibilities: that strategy may retain its claim on $\sigma$ forever, or the claim on $\sigma$ may pass to a higher priority strategy.  As we will argue, each $R_e$ strategy changes its directive only finitely many times.  Thus we may speak of a strategy's ultimate directive.  We then define $X$ as follows, for each string $\sigma$:
\begin{itemize}
    \item If $\sigma \not \in N_U$, then $\sigma \not \in X$;
    \item If $\sigma \in N_U$ but $\sigma$ is never claimed by any strategy, then $\sigma \not \in X$;
    \item If $\sigma \in N_U$ and $\sigma$ is claimed by some strategy, fix the highest priority strategy to ever claim $\sigma$.  If that strategy's ultimate directive is {\em meet}, then $\sigma \in X$; otherwise, $\sigma \not \in X$.
\end{itemize}

\smallskip

{\em Verification:}

First we must keep our promises.

\begin{claim}
At any stage $s$, for the strategy $R_e$, the value $w = w(e,s)$ is at most $2^{-(2e+3)}$.
\end{claim}

\begin{proof}
At each stage, $R_e$ claims at most one string, and that string will always have length at least $2e+4$.  So $w(e,s+1) - w(e,s) \le 2^{-(2e+4)}$.  Further, $R_e$ only claims a string if $w(e,s) \le 2^{-(2e+4)}$, so $w(e,s+1)$ is at most $2^{-(2e+4)} + 2^{-(2e+4)} = 2^{-(2e+3)}$.
\end{proof}

\begin{claim}
The sum $\sum 2^{-|\sigma|}$ over all strings $\sigma$ which are ever claimed in the course of the construction is at most $1/2$.
\end{claim}

\begin{proof}
If a string $\sigma$ is claimed during the construction, there are three possible fates for $\sigma$: 1) there is some $n$ such that $P_n$ claims $\sigma$ at almost every stage; 2) there is some $e$ such that $R_e$ claims $\sigma$ at almost every stage; 3) $\sigma$ leaves $N_U$.  We consider each case in turn.

Fix $n$.  By construction, there is at most one string which is ultimately claimed by $P_n$, and such a string has length at least $2n+4$.  So the sum $\sum 2^{-|\sigma|}$ over all strings $\sigma$ ultimately claimed by $P_n$ is bounded by $2^{-(2n+4)}$.

Fix $e$.  Let $\hat{C}$ be the set of strings $\sigma$ such that $R_e$ claims $\sigma$ at almost every stage.  Then each $\sigma \in \hat{C}$ contributes to almost every $w(e,s)$, so $\sum_{\sigma \in \hat{C}}2^{-|\sigma|} \le \sup_s w(e,s) \le 2^{-(2e+3)}$.

Finally, consider strings which leave $N_U$.  We will split this case into two subcases, based on whether the given string extends $000$, and so was introduced by our copying of $V$, or it extends 1, and so was introduced by our other actions.  Since the extensions of $000$ in the domain of $U$ form an antichain, the sum $\sum 2^{-|\sigma|}$ over all strings in the first subcase is bounded by $1/8$.

Consider now the second subcase.  Note that strings never leave $N_U$ because of our action; instead, if $\sigma$ leaves $N_U$, then there must be some $\tau \in N_V$ with $V(\tau) = U(\sigma)$ and $|\tau| < |\sigma|-3$.  As we always choose our values large, if $\sigma_0$ and $\sigma_1$ are distinct strings from this subcase, $U(\sigma_0) \neq U(\sigma_1)$, so the corresponding $\tau$s are distinct.  So summing over the $\sigma$ of this subcase, we have $\sum 2^{-|\sigma|} < \frac18 \sum_{\tau \in N_V} 2^{-|\tau|} < \frac18$.
Putting these all together, our desired sum is bounded by
\[
\sum_{n \in \omega} 2^{-(2n+4)} + \sum_{e \in \omega} 2^{-(2e+3)} + \frac18 + \frac18 \le \frac12,
\]
as desired.
\end{proof}

\begin{claim}
Each $R_e$ changes its directive only finitely many times.
\end{claim}

\begin{proof}
Suppose not.  Then there is a sequence of stages $s_0 < s_1 < \dots$ such that $R_e$ has directive {\em avoid} at stage $s_i$, and has directive {\em meet} at stage $s_i+1$, for every $i$.  In order to switch from {\em meet} at stage $s_i+1$ back to {\em avoid} at stage $s_{i+1}$, every string claimed by $R_e$ at stage $s_i+1$ must either be stolen by a higher priority strategy or leave $N_U$, both of which are irreversible.  Thus the strings which contribute to $w(e,s_i+1)$ must be entirely different from those which contribute to $w(e, s_j+1)$ for $j \neq i$.  But $w(e,s_i+1) > 2^{-(2e+4)}$ for every $i$, and the strings which contribute to $w(e,s_i+1)$ all belong to $\={dom}(U)$, a contradiction.
\end{proof}

Our promises being met, the construction of $X$ is as described.  Now  we verify that $X$ has the desired properties.

\begin{claim}
For each $n$, the strategy $P_n$ ultimately claims a string which it never renounces its claim upon, and thus $X$ is infinite.
\end{claim}

\begin{proof}
Fix $n$.  It suffices to argue that there is some string in $N_U$ of length at least $2n+4$ which is never claimed by any strategy.

Fix $s_0$ such that $N_{s_0}$ has converged on all strings of length less than $2n+4$.  We build a sequence $\sigma_0, \sigma_1, \dots \in N_U$:
\begin{itemize}
    \item Fix some $\sigma_0 \in N_U \setminus \={dom}(U_{s_0})$.  Such a $\sigma_0$ must exist, as $U$ is universal (and in particular, surjective).
    \item If $\sigma_i$ is eventually claimed by some strategy, then the construction responds by enumerating a $\tau$ into $\={dom}(U)$ with $|\tau| = |\sigma_i|$ and $U(\tau)$ not any previously seen value.  It may be that $\tau$ is not a minimal code, but there is some $\sigma_{i+1} \in N_U$ with $U(\sigma_{i+1}) = U(\tau)$, and $\sigma_{i+1}$ enters $\={dom}(U)$ after the stage at which $\sigma_i$ is claimed (by the largeness of $U(\tau)$).
\end{itemize}
Inductively, we see that $U(\sigma_i) \neq U(\sigma_j)$ for any $j < i$, and so the sequence $\sigma_0, \sigma_1, \dots$ is injective.  Further, $|\sigma_{i+1}| \le |\sigma_i|$.  As $\sum_{\sigma \in N_U} 2^{-|\sigma|} < 1$, this sequence must be finite, so there is some $\sigma_i$ which is never claimed by any strategy.  Since $\sigma_i$ enters $\={dom}(U)$ after stage $s_0$, $|\sigma_i| \ge 2n+4$ by choice of $s_0$.
\end{proof}

\begin{claim}
$X$ is semi-low.
\end{claim}
\begin{proof}
We give an algorithm for determining whether $X \cap W_e = \emptyset$, using oracle $\emptyset'$.  First, $\emptyset'$ can determine a stage $s_0$ such that every $R_j$ strategy with $j \le e$ has settled on its ultimate directive, and such that each $P_n$ strategy with $n < e$ has made its ultimate claim.

We may ignore those $R_j$ with $j \le e$ which have {\em avoid} as their ultimate directive.  For the remaining, they have claimed some finitely many strings by stage $s_0$, and none will ever claim another string.  With oracle $\emptyset'$, we can examine the entire finite collection to determine if there is a string $\sigma$ among them which remains claimed by its current strategy forever, and with $\sigma \in W_e$.

We claim that there is such a $\sigma$ if and only if $X \cap W_e \neq \emptyset$.  In the one direction, if there is such a $\sigma$, then $\sigma \in X$ by construction, so $\sigma \in X \cap W_e$.

In the other direction, if there is no such $\sigma$, note that this implies that $R_e$'s ultimate directive is {\em avoid} -- if it were {\em meet}, then $R_e$'s strings are amongst those examined, so it must eventually renounce its claim to all of them, resulting in $R_e$ changing directive to {\em avoid}, contrary to choice of $s_0$.  Now for any $\tau \in W_e \cap N_U$, we consider two cases: $|\tau| < 2e+4$, and $|\tau| \ge 2e+4$.

If $|\tau| < 2e+4$, then $\tau$ is too short to be claimed by any strategy of lower priority than $R_e$, and by assumption $\tau$ cannot be ultimately claimed by any strategy of higher priority with ultimate directive {\em meet}.  So $\tau \not \in X$.

If $|\tau| \ge 2e+4$, then $\tau$ will eventually be claimed by $R_e$, by construction, and so $\tau \not \in X$.
\end{proof}
This completes the proof.
\end{proof}

\begin{cor}
For some, and hence any, universal prefix-free machine $U$,
\[
\limsup_{\tau\in N_U} K(\sigma\mid \tau) = \limsup_{\tau \in M_U} K(\sigma\mid \tau) =^+ K^{\emptyset'}(\sigma).
\]
\end{cor}

\begin{proof}
As this is independent of choice of machine, let $U$ and $X$ be as in \Cref{prop:semi_low_existence}.  Then
\[
K^{\emptyset'}(\sigma) =^+ \limsup_{\tau \in X} K(\sigma\mid \tau) \le \limsup_{\tau \in N_U} K(\sigma\mid \tau) \le \limsup_{\tau \in M_U} K(\sigma\mid \tau) \le \limsup_{\tau \in 2^{<\omega}} K(\sigma\mid \tau) =^+ K^{\emptyset'}(\sigma),
\]
where the first equality is by \Cref{prop:semi-lowness_suffices}, the last is by \Cref{thm:motivation}, and the inequalities are by subset.
\end{proof}

\section{Where the limsup's live, and finite strings as oracles}
\label{misc}

Here we collect some miscellaneous results about finite strings as oracles.  The first is motivated by our (numerous!)\ failed attempts to prove \cref{thm:main_result} before we finally discovered the method of \cref{semi,main}.  

One avenue we pursued was attempting to determine for which $m$ does $K(\sigma\mid m)$ achieve $\limsup_n K(\sigma\mid n)$.  A natural candidate is the nondeficiency stages: fix a computable enumeration $(a_m)_{m \in \omega}$ of $\emptyset'$, and define
\[
E = \{ m : (\forall n > m)\, [a_m < a_n]\}.
\]
This is the basis for the method of {\em true stages} 
(see Montalb{\'a}n \cite{montalban}
for a modern interpretation for higher level 
priority arguments, but the idea going back to Dekker \cite{De},
as per Soare \cite{Soare}, Ch.\ V 2.5), where the elements of $E$ are employed because they make correct guesses about $\emptyset'$ (as we shall see in a moment).

However, this turns out to be approaching from the wrong direction.  Since $K^{\emptyset'}(\sigma) = \limsup_n K(\sigma\mid n)$, to find places where the limit supremum is achieved, we are not concerned with doing {\em as well as} $\emptyset'$ -- we are concerned with doing {\em no better than} $\emptyset'$.  Thus we are looking not for $n$ which are powerful, but for those which are weak.

The following result says that for $m \in E$, $K(\sigma\mid m)$ does much better than $\limsup_n K(\sigma\mid n)$.
\begin{prop}
\[
\limsup_{m \in E} K(\sigma \mid m) =^+ K^{\emptyset^{(2)}}(\sigma).
\]
\end{prop}

\begin{proof}
For any $m$, define $\tau_m \in 2^{<\omega}$ to be the string of length $a_m$ such that $\tau_m(x) = 1$ iff $x = a_n$ for some $n < m$.  Note that $m \mapsto \tau_m$ is effective, so $K(\sigma\mid m) \le^+ K(\sigma\mid \tau_m)$.  Also, $\sigma_m \prec \emptyset'$ iff $m \in E$, so
\begin{align*}
\limsup_{m \in E} K(\sigma \mid m) &\le^+ \limsup_{m \in E} K(\sigma \mid \emptyset'\uh a_m)\\ &\le \limsup_n K(\sigma \mid \emptyset'\uh n)\\ &=^+ \limsup_n K^{\emptyset'}(\sigma\mid n).
\end{align*}
Conversely, $\emptyset'$ can compute the increasing enumeration of $E$, $E = \{ b_0 < b_1 < \dots\}$, so $K^{\emptyset'}(\sigma\mid n) \le^+ K(\sigma \mid b_n)$, giving
\[
\limsup_{m \in E} K(\sigma \mid m) =^+ \limsup_n K^{\emptyset'}(\sigma\mid n).
\]
By \cref{thm:motivation} relativized to $\emptyset'$, this is (up to an additive constant) $K^{\emptyset^{(2)}}(\sigma)$.
\end{proof}

The reader might note the following somewhat paradoxical situation.
The natural proof to show that $K^{\emptyset'}(\sigma)\le ^+
\limsup_n K(\sigma\mid n)$ is to folly approximate 
$K^{\emptyset'}(\sigma)[n]$ at each stage $n$, where both the computations and 
oracles are approximated for $n$ stages.
We would do this as part of the computation of $K^{\overline{n}}(\sigma)$ 
for some machine $M^{\overline{n}}(\sigma)$ via the Machine Existence Theorem
as mentioned above. 
Therefore, for all stages $t>n$ it can only be that 
$K(\sigma\mid n)[t]\le^+ K(\sigma\mid n)[n] \le^+ K^{\emptyset'}(\sigma)[n].$
The true value of $K^{\emptyset'}(\sigma)$ must have been achieved at a true 
stage, but we see above, it does not happen at almost all true stages.
Thus it must be achieved at infinitely many non-true stages $s$, but where
$K^{\emptyset'}[s]=^+ K^{\emptyset'}(\sigma).$ We don't really understand the 
characteristics of such ``almost true'' stages $s$. We also point out that 
the limsups appear to  be achieved for different $s$'s for different $\sigma$'s.

\cref{thm:motivation} says that for almost any string $\sigma$, almost any finite oracle can aid in the compression of $\sigma$.  A priori, however, there is no reason to expect there to be a single finite oracle which aids in the compression of almost every string.  Nevertheless, this is the case.

\begin{thm}\label{thm:e-compressing}
For all $e$ there is a string $\rho$ such that for almost all $\tau$,
$K(\tau\mid \rho)<K(\tau)-e.$
$K^\xi(\tau)<K(\tau)-e.$
\end{thm}

We give two proofs.  The first is based on Symmetry of Information and \cref{thm:motivation}.

\begin{proof}[First proof of \cref{thm:e-compressing}]
 By \cref{thm:motivation},
$K^{\emptyset'}(\sigma)=^+ \limsup_n K(\sigma\mid n).$
The proof shows that
$K^{\emptyset'}(\sigma)=^+ \limsup_{\{\tau\mid |\tau|\to \infty\}}
K(\sigma\mid \tau)$. (Alternatively use $1\tau$ for any string $\tau$,
with a $=^+$, change.)
By a counting argument, we know that as $|\tau|\to \infty$, 
$|\tau^*|\to \infty$.

Solovay \cite{So} (see Downey and Hirschfeldt \cite{DH}, Ch 10.2, Lemma 10.2.6)
proved that
\[
K^{\emptyset'} (n)\le K(n)-\alpha(n)+O(\log \alpha(n)),
\]
where 
$\alpha(n)=\min \{K(m)\mid m\ge n\}$.

Thus, for all $e$, 
and almost all $\sigma$,
\[
K^{\emptyset'}(\sigma)<K(\sigma)-e.
\]
Hence, for all 
$e$ there is an $m$  and  $\sigma=n$
such that for all $\tau$ with
$|\tau|>m$,
\[
K(\sigma\mid \tau^*)<K(\sigma)-e.
\]

Now, Symmetry of Information (Levin  and G{\'a}cs \cite{Ga}, Chaitin \cite{58})
says that
\[
K(\sigma,\tau)=K(\sigma)+K(\tau\mid \sigma^*)=K(\tau)+K(\sigma\mid \tau^*).
\]
Thence,
for $\sigma$ as above and for $\tau$ with $|\tau|>m$,
we have 
\[
K(\sigma)+K(\tau\mid \sigma^*)=K(\tau)+K(\sigma\mid \tau^*).
\]
Since $K(\sigma\mid \tau^*)<K(\sigma)-e,$
this gives 
$$K(\sigma)+K(\tau\mid \sigma^*)<K(\tau)+K(\sigma)-e.$$
Thus,
$K(\tau\mid \sigma^*)<K(\tau)-e.$ Choosing $\sigma^*=\rho$ gives the result.
\end{proof}

Our second proof is based on the conditional complexity variant of the Machine Existence Theorem, which we first state.

\begin{prop}
Suppose $A \subseteq 2^{<\omega} \times \omega \times 2^{<\omega}$ is a c.e.\ set such that for every $\tau \in 2^{<\omega}$, $\sum_{(\sigma, s, \tau) \in A} 2^{-s} \le 1$.  Then for all $(\sigma, s, \tau) \in A$, $K(\sigma\mid \tau) \le^+ s$.
\end{prop}

\begin{proof}[Second proof of \cref{thm:e-compressing}]
We will enumerate a c.e.\ set $A$.  Fix an effective bijection $\tau = \seq{D, k}$ between $\tau \in 2^{<\omega}$ and pairs $\seq{D, k}$ with $D \subset 2^{<\omega}$ finite and $k \in \omega$.  For $\tau = \seq{D, k}$, for every $\sigma \in 2^{<\omega} \setminus D$, we enumerate $(\sigma, t-k, \tau)$ into $A$ for every $t \ge K(\sigma)$, provided doing so does not cause $\sum_{(\sigma, s, \tau) \in A} 2^{-s}$ to exceed 1.

Fix the constant $c$ such that $K(\sigma\mid \tau) \le s+c$ for every $(\sigma,s, \tau) \in A$.  As $\sum_\sigma 2^{-K(\sigma)} < 1$, there is some finite $D$ such that $\sum_{\sigma \not \in D} 2^{-K(\sigma)} < 2^{-(e+c+1)}$.  Fix $\tau = \seq{D, e+c}$.  Then $\sum_{\sigma \not \in D} \sum_{t \ge K(\sigma} 2^{-(t - e - c)} < 1$, so $(\sigma, t-e-c, \tau)$ is successfully enumerated into $A$ for all such $\sigma$ and $t$.  Thus $K(\sigma\mid\tau) \le K(\sigma)  - e - c+c = K(\sigma)-e$ for all $\sigma \not \in D$.
\end{proof}

Define $\rho$ to be $e$-\emph{compressing}
if for almost all $\tau$,
$K(\tau\mid \rho)<K(\tau)-e$. 

\begin{qstn}
What can be said about the set 
$C_e=\{\rho\mid \rho $ is $e$-compressing$\}$?
\end{qstn}

\end{document}